\documentclass[oneside,notitlepage, a4paper, 12pt]{amsart}
\usepackage[leqno]{amsmath}

\usepackage{amsfonts,amscd,latexsym,amsmath,amssymb,cancel,enumerate,mathrsfs,color,hyperref,graphicx}
\theoremstyle{plain}
\newtheorem{master}{Master}[section]
\newtheorem{prop}[master]{Proposition}
\newtheorem{thm}[master]{Theorem}
\newtheorem{fact}[master]{Fact}
\newtheorem{lem}[master]{Lemma}
\newtheorem{cor}[master]{Corollary}
\newtheorem{question}[master]{Question}
\newtheorem{problem}[master]{Problem}

\theoremstyle{definition}
\newtheorem{defin}[master]{Definition}

\theoremstyle{remark}
\newtheorem{remark}[master]{Remark}
\numberwithin{equation}{section}
\newcommand{\Ur}{\mathbb{U}}
\newcommand{\Rea}{\mathbb{R}}
\newcommand{\Nat}{\mathbb{N}}
\newcommand{\Rat}{\mathbb{Q}}
\newcommand{\Int}{\mathbb{Z}}
\newcommand{\Age}{\mathcal{K}}

\newcommand{\Iso}{\mathrm{Iso}}
\newcommand{\Act}{\mathrm{Act}}
\newcommand{\Hom}{\mathrm{Hom}}

\newcommand{\Norm}{\|\cdot\|}
\newcommand{\GL}{\mathrm{GL}}

\newcommand{\Class}{\mathcal{C}}

\begin{document}
\title[Universal actions]{Universal actions and representations of locally finite groups on metric spaces}
\author{Michal Doucha}
\address{Institute of Mathematics, Academy of Sciences, Prague, Czech republic}
\email{doucha@math.cas.cz}
\thanks{The author was supported by the GA\v CR project 16-34860L and RVO: 67985840.}
\begin{abstract}
We construct a universal action of a countable locally finite group (the Hall's group) on a separable metric space by isometries. This single action contains all actions of all countable locally finite groups on all separable metric spaces as subactions. The main ingredient is the amalgamation of actions by isometries. We show that an equivalence class of this universal action is generic.

We show that the restriction to locally finite groups in our results is necessary as analogous results do not hold for infinite non-locally finite groups.

We discuss the problem also for actions by linear isometries on Banach spaces.
\end{abstract}
\maketitle
\section*{Introduction}
Groups acting by isometries on metric and Banach spaces are one of the active areas of research in geometry, group theory and functional analysis. In this paper, we are interested in amalgamation of group actions and constructing universal actions on metric spaces and universal representations in Banach spaces. It is well known from the beginnings of combinatorial group theory that one can construct an amalgam of two groups over some common subgroup. At least as old is the amalgamation of metric spaces, or amalgamation of normed vector spaces. However, to the best of our knowledge, nobody has considered yet amalgamation of actions of groups on metric or Banach spaces by isometries. In metric geometry or functional analysis, amalgamation techniques are often used to construct various universal metric or Banach spaces (consider for instance the Urysohn universal metric space \cite{Ur}, or the Gurarij universal Banach space \cite{Gu}). The well-known Hall's universal locally finite group (\cite{Ha}) is essentially made by amalgamating finite groups. Here by amalgamating actions of finite groups on finite metric spaces by isometries we obtain the following result.
\begin{thm}\label{intro_mainthm}
There exists a universal action of the Hall's locally finite group $G$ on the Urysohn space $\Ur$ by isometries. That is, for any action of a countable locally finite group $H$ on a separable metric space $X$ by isometries, there exists a subgroup $H'\leq G$ isomorphic to $H$ such that, after identifying $H$ and $H'$, there is an $H$-equivariant isometric embedding of $X$ into $\Ur$.
\end{thm}
The meaning of the theorem is that there is a single action of a countable locally finite group on a separable metric space by isometries that contains all actions of all countable locally finite groups on all separable metric spaces as subactions.

One of the main ingredients is the amalgamation of actions and we have the following general theorem.
\begin{thm}
Let $G_1,G_2$ be two groups (countable or not) with a common subgroup $G_0$. Suppose that $G_1$ acts on a metric space $X_1$ and $G_2$ acts on $X_2$, by isometries in both cases. Let $X_0$ be a common $G_0$-invariant subspace of $X_1$ and $X_2$, i.e. the restrictions of the two actions on $G_0$ and $X_0$ coincide. Then there is an amalgam of the action, which is an action of $G_1\ast_{G_0} G_2$ on a metric space of density character $\max\{|G_1\ast_{G_0} G_2|,\mathrm{dens}(X_1),\mathrm{dens}(X_2)\}$.
\end{thm}
Following the research of Rosendal in \cite{Ro} and of Glasner, Kitroser and Melleray in \cite{GKM} we investigate the genericity of the universal action from Theorem \ref{intro_mainthm}. 
\begin{thm}
The universal action from Theorem \ref{intro_mainthm} is \emph{weakly generic} in some sense. That is, the set of those actions in the Polish space $\Hom(G,\Iso(\Ur))$ that are naturally equivalent to the universal one is dense $G_\delta$.
\end{thm}
As a consequence, we derive the following result which was originally asked by Melleray and Tsankov (in \cite{MeTs}) for abelian groups.
\begin{thm}
There exists a Polish group $\mathbb{H}$ such that for comeager many actions $\alpha\in\Hom(G,\Iso(\Ur))$ we have that the closure $\overline{\alpha[G]}$ is topologically isomorphic to $\mathbb{H}$.
\end{thm}
We show that the restriction to locally finite groups in our results is essential.
\begin{thm}
There are no analogously universal actions of infinite groups that are not locally finite.
\end{thm}
Moreover, we show that the amalgamation of the actions does not work in the abelian category.
\begin{thm}
The class of actions of finite abelian groups on finite metric spaces does not have the amalgamation property.
\end{thm}
Finally, we discuss universal actions on Banach spaces. General actions by isometries are by affine isometries. Unfortunately, we show that no universal action by affine isometries can exist, even of finite groups. Thus we are forced to restrict to actions by linear isometries, i.e. representations in Banach spaces. We propose a class of actions of finite groups on finite-dimensional Banach spaces such that, provided this class has the amalgamation property, its Fra\" iss\' e limit would be a universal action of the Hall's group on the Gurarij space by linear isometries.
\section{Preliminaries}
Let us start with our notational convention. All the group actions in this paper are by isometries. We usually denote actions by the symbol `$\alpha:G\curvearrowright X$', where $G$ is a group and $X$ is a metric space. However, as it is common, we usualy write $g\cdot x$ instead of $\alpha(g,x)$.

Regarding groups, we are mostly concerned with locally finite ones, where a group is locally finite if every finitely generated subgroup is finite. Since we shall work solely with countable groups, it is the same as saying that the group is a direct limit of a sequence of finite groups.

Our constructions of universal objects are based on techniques commonly referred as ``Fra\" iss\' e theory". We refer to Chapter 7 in \cite{Ho} for more information about this subject. For a reader unfamiliar with this method we briefly and informally describe the basics of Fra\" iss\' e theory that we use in the paper.

Let $\Age$ be some countable class of mathematical objects of some type with some notion of embedding between these objects. Suppose that direct limits of objects from $\Age$ exist. Think of the class of finite groups for instance. We say it is a \emph{Fra\" iss\' e class} if any two objects from $\Age$ can be embedded into a single object from $\Age$, such a property is called \emph{joint embedding property}, and if whenever we have objects $A,B,C\in \Age$ such that $A$ embeds into both $B$ and $C$, witnessed by embeddings $\iota_B$, resp. $\iota_C$, then there exists an object $D\in\Age$ and embeddings $\rho_B$, resp. $\rho_C$ of $B$ into $D$, resp. $C$ into $D$ such that $\rho_C\circ\iota_C=\rho_B\circ\iota_B$; i.e we can do amalgamation with object from $\Age$. The latter property is called \emph{amalgamation property}. The Fra\" iss\' e theorem (see Chapter 7 in \cite{Ho}) then asserts that there exists a unique object $K$, called the \emph{Fra\" iss\' e limit of $\Age$}, which is a direct limit of a sequence of objects from $\Age$ satisfying
\begin{itemize}
\item every object $A\in\Age$ embeds into $K$;
\item whenever we have objects $A,B\in \Age$ such that $A$ embeds via $\rho_A$ into $K$ and via $\iota_A$ into $B$, then there exists an embedding $\rho_B$ of $B$ into $K$ such that $\rho_A=\rho_B\circ\iota_A$.

\end{itemize}
The second property is called the \emph{extension property} and will be used in our proofs of universality of certain actions. Note that whenever $X$ is some direct limit of a sequence of objects from $\Age$, then successive application of the extension property gives an embedding of $X$ into $K$.

We note that the Fra\" iss\' e theorem stated above is the only tool which we shall use and its proof is actually shorter that the discussion on Fra\" iss\' e theory above and may be left as an exercise. Since we are going to work with Fra\" iss\' e classes which are `metric' we note that recently a general theory for metric Fra\" iss\' e classes was developed independently in \cite{BY} and \cite{Ku}. However, we shall not directly use their results in our paper.\\

\noindent {\bf Example 1} Consider the countable class of all finite graphs. It is easy to show it has the joint and amalgamation properties, thus by the Fra\" iss\' e theorem there exists a Fra\" iss\' e limit, a certain direct limit of a sequence of finite graphs, which is a countable graph commonly known as the random graph, or the Rado graph. The extension property allows to show that it contains as a subgraph a copy of every countable graph.\\

\noindent {\bf Example 2} Consider now the countable class of all finite abelian groups. It is again easy to show the joint and amalgamation properties and one can even show that the Fra\" iss\' e limit is nothing else than $\bigoplus_{n\in\Nat} \Rat/\Int$.\\

\noindent {\bf Example 3} Consider now the countable class of all finite groups, not necessarily abelian. This is the most important example for us regarding the topic of our paper. It is less straightforward, nevertheless possible to show (see \cite{Ne}), that this class has the amalgamation property, and thus also the joint embedding property. The Fra\" iss\' e limit is what is commonly known as the Hall's universal locally finite group (\cite{Ha}).\\

\noindent {\bf Example 4} Consider the countable class of all finitely presented groups. It is again easy to show the amalgamation property. We are not aware anyone has considered the Fra\" iss\' e limit of this class yet.\\

\noindent {\bf Example 5} Consider the countable class of all finite metric spaces with rational distances. The amalgamation and joint embedding is again straightforward. The Fra\" iss\' e limit is what is known as the rational Urysohn space. Its completion is the Urysohn universal space (see \cite{Ur}).\\

\noindent {\bf Example 6} As the last example, we present another `metric Fra\" iss\' e class' recently discovered by the author in \cite{Do}. It is the class of all finitely generated free abelian groups with a `finitely presented rational metric'. The completion of its limit gives the metrically universal abelian separable group. See the paper for details.
\section{Universal actions}
Let us start with the discussion on the notion of universality, which can be naturally done in the category-theoretical language. Given a category, consisting of objects and morphism (or embeddings) between them, an object is universal if for every object from the category there is a morphism (embedding) into the universal one. Our objects are groups acting on metric spaces by isometries. If the acting group, say $G$, is fixed, the natural notion of embedding is that of `$G$-equivariant isometric embedding'. If the acting group is allowed to vary, then the embedding should consist of both group monomorphism and equivariant isometric embedding. We propose two notions of universality based on these two choices of embeddings.
\begin{defin}\label{def_univaction}
Let $\Class$ be a class of actions of countable groups on separable metric spaces by isometries. Say that  $\alpha: G\curvearrowright X\in\Class$ is a universal action from $\Class$ if for any action $\beta:H\curvearrowright Y\in\Class$ there is a subgroup $H'\leq G$ isomorphic to $H$ and an isometric embedding of $Y$ into $X$ which is, after identifying $H$ and $H'$, $H$-equivariant.
\end{defin}
The previous notion of universality corresponds to the universality from Theorem \ref{intro_mainthm} if one takes as $\Class$ the class of all actions of all countable locally finite groups on all separable metric spaces.

We however state another notion of universality which is natural. As it will turn out, it is too strong.
\begin{defin}\label{def_univaction2}
Let $G$ be a fixed countable group. Let $\Class$ be a class of actions of $G$ on separable metric spaces. Say that $\alpha: G\curvearrowright X\in\Class$ is a universal action from $\Class$ if for any action $\beta:G\curvearrowright Y\in\Class$ there is an $G$-equivariant isometric embdedding of $Y$ into $X$.
\end{defin}
\subsection{Proof of the main theorem}
In this subsection, we define a natural Fra\" iss\' e class of actions of finite groups on finite metric spaces. Using that, we will prove Theorem \ref{intro_mainthm}.
\begin{defin}
Let $G$ be a group and $X$ a metric space. A \emph{pointed free action} of $G$ on $X$ by isometries is a tuple $(G\curvearrowright X, (x_i)_{i\in I})$, where $G\curvearrowright X$ is a free action of $G$ on $X$ by isometries and $I$ is some index set for the orbits of the action and $(x_i)_{i\in I}$ is a selector on the orbits, i.e. $X=\bigcup_{i\in I} G\cdot x_i$ and for $i\neq j$, $x_i$ and $x_j$ lie in different orbits.
\end{defin}
There is also a natural notion of an embedding between two pointed free actions. Suppose we are given two such actions $(H\curvearrowright Y,(y_i)_{i\in I})$ and $(G\curvearrowright X, (x_j)_{j\in J})$.  An \emph{embedding} of $(H\curvearrowright Y,(y_i)_{i\in I})$ into $(G\curvearrowright X, (x_j)_{j\in J})$ is a pair $(\phi,\psi)$, where $\phi:H\hookrightarrow G$ is a group embedding and $\psi:Y\hookrightarrow X$ is an isometric embedding that sends the distinguished points $(y_i)_{i\in I}$ into the set of distinguished points $(x_j)_{j\in J}$ and such that for any $i,j\in I$ and $f,h\in H$ we have $$d_Y(f\cdot y_i,h\cdot y_j)=d_X(\phi(f)\cdot \psi(y_i),\phi(h)\cdot \psi(y_j)).$$

Notice that in the case when the group monomorphism $\phi$ is just an inclusion, the previous definition says that $\psi$ is an $H$-equivariant isometric embedding that sends the set of distinguished points into the set of distinguished points.
\begin{thm}\label{thm_amal}
The pointed free actions can be amalgamated.
\end{thm}
\begin{remark}
It means that for any embeddings $\psi_i:(G_0\curvearrowright X_0,(x_j)_{j\in I_0})\hookrightarrow (G_i\curvearrowright X_i,(x_j)_{j\in I_i})$, for $j\in\{1,2\}$, where we assume that $G_0\leq G_1$ and $G_0\leq G_2$, there are a group $G_1,G_2\leq G_3$, pointed action $(G_3\curvearrowright X_3, (x_j)_{j\in I_3})$  and embeddings $\rho_j:(G_i\curvearrowright X_i,(x_j)_{j\in I_i})\hookrightarrow (G_3\curvearrowright X_3, (x_j)_{j\in I_3})$, for $j\in \{1,2\}$, such that $\rho_2\circ \psi_2=\rho_1\circ \psi_1$.
\end{remark}
\begin{proof}
Consider such actions from the remark above, i.e. $(G_i\curvearrowright X_i,(x_j)_{j\in I_i})$, for $i\in\{0,1,2\}$. We may also suppose that $I_0\subseteq I_i$, for $i=1,2$, and that $I_0=I_1\cap I_2$. Let $G_3$ be $G_1\ast_{G_0} G_2$, i.e. the free product of $G_1$ and $G_2$ amalgamated over $G_0$ (we refer to \cite{LySch} for constructions of amalgamated free products of groups). Let $I_3=I_1\cup I_2$ and set $X_3=\bigcup_{j\in I_3} G_3\cdot j$. Clearly, $X_i\subseteq X_3$, for $i=1,2$. We shall define a metric on $X_3$ so that the canonical action of $G_3$ on $X_3$ is by isometries and that the inclusion of $X_i$ into $X_3$ is isometric (it is obviously $G_i$-equivariant), for $i=1,2$.

We define a structure of a weighted graph on $X_3$ that will help us define a metric there. That is, we define edges on $X_3$ and then associate a certain weight function $w$ giving positive real numbers to these edges. For $g,h\in G_3$ and $i,j\in I_3$, the elements $g\cdot x_i$ and $h\cdot x_j$ are connected by an edge if and only if
\begin{itemize}
\item either $g^{-1}h\in G_1$ and $i,j\in I_1$, then its weight is $$w(g\cdot x_i,h\cdot x_j)=d_{X_1}(x_i,g^{-1}h\cdot x_j);$$
\item or $g^{-1}h\in G_2$ and $i,j\in I_2$, then analogously its weight is  $$w(g\cdot x_i,h\cdot x_j)=d_{X_2}(x_i,g^{-1}h\cdot x_j).$$
\end{itemize}
In case that $g^{-1}h\in G_0$ and $i,j\in I_0$ there is no ambiguity in the definition. Indeed, by assumption, in such a case we have $$d_{X_0}(x_i,g^{-1}h\cdot x_j)=d_{X_1}(x_i,g^{-1}h\cdot x_j)=d_{X_2}(x_i,g^{-1}h\cdot x_j).$$
It is clear that this graph is connected, so we define the graph metric $d$ on $X_3$ as follows: for $x,y\in X_3$ we set $$d(x,y)=\inf\{\sum_{i=1}^n w(e_i):e_1\ldots e_n\text{ is a path from }x\text{ to }y\}.$$ In case the groups and the index sets are finite we may replace the infimum above by minimum.
It follows immediately from the definition that the natural action of $G_3$ on $X_3$ is a weighted graph automorphism, i.e. it preserves the edges including their weight. It follows that $G_3$ acts by isometries on $X_3$. We shall check that the canonical embeddings (inclusions) of $X_1$ and $X_2$ into $X_3$ are isometric.

We shall check it for both $X_1$ and $X_2$. Thus fix some $g,h\in G$ and $i,j\in I_3$ such that either both $g,h\in G_1$ and both $i,j\in I_1$, or both $g,h\in G_2$ and both $i,j\in I_2$. We need to check that $d_{X_l}(g\cdot x_ih\cdot x_j)=d(g\cdot x_i,h\cdot x_j)$, where $l\in \{1,2\}$ depending on whether $g,h\in G_1$, $i,j\in I_1$, or $g,h\in G_2$, $i,j\in I_2$. It is clear that $d_{X_l}(g\cdot x_i,h\cdot x_j)\geq d(g\cdot x_i,h\cdot x_j)$, so suppose there is a strict inequality and we shall reach a contradiction. There is then an edge-path $e_1\ldots e_n$ from $x=g\cdot x_i$ to $y=h\cdot x_j$. By induction on $n$, the length of the path, we shall show that  $d_{X_l}(g\cdot x_i,h\cdot x_j)\leq w(e_1)+\ldots+w(e_n)$. The case $n=1$ is clear, so we suppose that $n\geq 2$ and we have proved it for all paths of length strictly less than $n$ between all pairs of elements from $X_1$ and all pairs of elements from $X_2$.

Now without loss of generality we suppose that $g,h\in G_1$, $i,j\in I_1$, the other case is analogous. For $1\leq l\leq n$, let $z_l=g_l\cdot x_{i_l}$ be the start vertex of $e_l$ and $z_{l+1}=g_{l+1}\cdot x_{i{_l+1}}$ the end vertex. Set $h_l=g_l^{-1}g_{l+1}$, for $1\leq l\leq n$. It follows that $gh_1 h_2\ldots h_n=h$ and each $h_l$ belongs to either $G_1$ or $G_2$. If all the $h_l$'s belong to $G_1$ then also all the $i_l$'s belong to $I_1$ and the path goes within $X_1$ and we can use the triangle inequalities there. So we suppose that some $h_l$, $1\leq l\leq n$, is from $G_2$; equivalently, that the path leaves $X_1$ at some point. Let $1\leq l<n$ be the least index where the path leaves $X_1$, i.e. $z_l\in X_1$, while $z_{l+1}\notin X_1$. It follows that $i_l\in I_0$. Indeed, by assumption for all $k\leq l$ we have $i_k\in I_1$, however since $h_l\in G_2$, by the definition of the edges of the graph we must have also that $i_l\in I_2$; thus $i_l\in I_1\cap I_2=I_0$. Now let $l<l'\leq n$ be the least index such that the path returns back to $X_1$, i.e. the least index $l<l'$ such that $z_{l'}\in X_1$. Again necessarily $i_{l'}\in I_0$. If $1<l$ or $l'<n$, then the subpath $e_l\ldots e_{l'-1}$ between two elements of $X_1$ is strictly shorter than $n$ and thus by the inductive hypothesis we have $d_{X_1}(z_l,z_{l'})\leq w(e_l)+\ldots+w(e_{l'-1})$. So we may replace this subpath by a single edge going from $z_l$ to $z_{l'}$, hereby again shortening the path, so by the inductive hypothesis we get $d_{X_1}(g\cdot x_i,h\cdot x_j)\leq w(e_1)+\ldots+w(e_n)$.\\

Thus we are left with the case $l=1$ and $l'=n$. In such a case we have $h_1\in G_2$ and $h_n\in G_2$, $i_1,i_2,i_{n-1},i_n\in I_0$, and also, since $n$ is the least number $l$ such that $h_1\ldots h_l\in G_1$, we must actually have $h_1\ldots h_n\in G_0$. It follows that $g$ and $h$ lie in the same left-coset of $G_0$ in $G_1$, i.e. $g^{-1}h\in G_0$. It follows that $d(g\cdot x_i,h\cdot x_j)=d(x_i,g^{-1}h\cdot x_j)$. Thus it suffices to show that $$d(x_i,g^{-1}h\cdot x_j)=d_{X_0}(x_i,g^{-1}h\cdot x_j)=d_{X_1}(x_i,g^{-1}h\cdot x_j),$$ where the latter equality is known and we need to show the former. In other words, we shall thus now, without loss of generality, assume that $g=1$, so $h=h_1\ldots h_n\in G_0$ and $x_i, h\cdot x_j\in X_0$.

We have two cases:
\begin{enumerate}
\item If $n=2$, i.e. $h=h_1 h_2$, then the path $e_1 e_2$ is within $X_2$ between two elements from $X_0$. Therefore, by the triangle inequality in $X_2$, its length is greater or equal to the path consisting of a single edge from $x_i$ to $h\cdot x_j$, that means we have $$w(e_1)+w(e_2)=d_{X_2}(x_i,h_1\cdot x_{i_2})+d_{X_2}(h_1\cdot x_{i_2},h\cdot x_j)\geq$$ $$d_{X_2}(x_i,h\cdot x_j)=d_{X_0}(x_i,h\cdot x_j)=d_{X_1}(x_i,h\cdot x_j),$$ and we are done.
\item If $n>2$, then the non-trivial subpath $e_2\ldots e_{n-1}$ is a path of length strictly less than $n$ between two elements from $X_2$ (note that $z_2=h_1\cdot x_{i_2}\in X_2$ and also $z_n=h h^{-1}_n\cdot x_{i_n}\in X_2$), thus by the inductive hypothesis we get that $$w(e_2)+\ldots+w(e_{n-1})\geq d_{X_2}(z_2,z_n).$$ It follows that $$\sum_{l=1}^n w(e_l)\geq d_{X_2}(x_i,z_2)+d_{X_2}(z_2,z_n)+d_{X_2}(z_n,h\cdot x_j)\geq$$ $$d_{X_2}(x_i,h\cdot x_j)=d_{X_0}(x_i,h\cdot x_j)=d_{X_1}(x_i,h\cdot x_j),$$ and we are again done.\\
\end{enumerate}
\end{proof}
\begin{remark}
The previous theorem was stated and proved for free actions. However, the proof can be modified to work for non-free actions as follows: Replace the metric by a pseudometric so that the action becomes free. Then proceed completely analogously working with pseudometrics instead of metrics and at the end make a metric quotient.
\end{remark}
\begin{remark}
We were informed by the referee that the proof of the previous theorem is related to the constructions from \cite{DaGu} where the authors show (among other things) that countable discrete groups that are uniformly embeddable into a Hilbert space are closed under taking amalgamated free products.
\end{remark}
In the next theorem we shall restrict our attention to actions of finite groups on finite metric spaces. In the proof, it will turn out that the theorem is a corollary of Theorem \ref{thm_amal}. That means that even if we are interested only in the theorem that follows it is natural to first prove Theorem \ref{thm_amal} and use the existence of the general amalgam to show the existence of the finite amalgam.
\begin{thm}
The class of pointed free actions of finite groups on finite metric spaces has the amalgamation property
\end{thm}
\begin{proof}
Let us start as in the previous theorem with three pointed actions  $(G_i\curvearrowright X_i,(x_j)_{j\in I_i})$, for $i\in\{0,1,2\}$ such that $G_0\leq G_i$, $I_0\subseteq I_i$, for $i=1,2$, $I_0=I_1\cap I_2$. Now the difference is that all the sets are finite. Let $G$ be now any amalgam group of $G_1$ and $G_2$ over $G_0$, e.g. the free product with amalgamation $G_1\ast_{G_0} G_2$. Set $I_3=I_1\cup I_2$ and $X_G=\bigcup_{j\in I_3} G\cdot x_j$. As in the proof of Theorem \ref{thm_amal} we define a weighted graph structure on $X_G$. That is, for $g,h\in G$ and $i,j\in I_3$, the elements $g\cdot x_i$ and $h\cdot x_j$ are connected by an edge if and only if
\begin{itemize}
\item either $g^{-1}h\in G_1$ and $i,j\in I_1$, then its weight is $$w(g\cdot x_i,h\cdot x_j)=d_{X_1}(x_i,g^{-1}h\cdot x_j);$$
\item or $g^{-1}h\in G_2$ and $i,j\in I_2$, then analogously its weight is  $$w(g\cdot x_i,h\cdot x_j)=d_{X_2}(x_i,g^{-1}h\cdot x_j).$$
\end{itemize}
There is no ambiguity when $g^{-1}h\in G_0$ and $i,j\in I_0$. We again define the graph metric as follows: for $x,y\in X_G$ we set $$d_G(x,y)=\min\{\sum_{i=1}^n w(e_i):e_1\ldots e_n\text{ is a path from }x\text{ to }y\}.$$ Notice that now $w$ assumes only finitely many values, so we may indeed use the minimum. Again, $G$ acts on $X_G$ by graph automorphisms preserving the weight function, thus also by isometries. In the proof of Theorem \ref{thm_amal} we showed that $d_G$ extends $d_{X_1}$ and $d_{X_2}$ in case $G=G_1\ast_{G_0} G_2$. We shall now find a finite amalgam $G$ with the same property. First set $G'=G_1\ast_{G_0} G_2$.

Set $M=\max\{w(e): e\text{ is an edge in }X_{G'}\}$ and $m=\min\{w(e):w(e)\neq 0\text{ and } e\text{ is an edge in }X_{G'}\}$. Set $K=\lceil \frac{M}{m}\rceil$. Consider the finite set $G_1\cup G_2$ as the set of generators of $G'$ and let $\lambda:G'\rightarrow [0,\infty)$ be the corresponding length function, i.e. the distance from the unit in $G'$ in the Cayley graph of $G'$ with $G_1\cup G_2$ as the generating set.

Now we use the fact that free amalgams of residually finite groups over finite groups are residually finite, see Theorems 2 and 3 in \cite{Ba}. We note that when the common subgroup is not finite, there are counterexamples (see \cite{Hig}). Thus let $G_3$ be a finite group such that there is an onto homomorphism $\phi:G'\rightarrow G_3$ which is injective on the ball $\{g\in G':\lambda(g)\leq K+1\}$. Clearly, $G_1$ and $G_2$ are subgroups of $G_3$ with the identified common subgroup $G_0$. Thus in particular, $G_3$ is a finite amalgamation of $G_1$ and $G_2$ over $G_0$. Moreover, we may suppose that $G_1\cup G_2$ generates $G_3$. Let $\rho$ be the length function on $G_3$ with respect to these generators. We have that $\phi$ is isometric with respect to $\lambda$ and $\rho$ on the ball $\{g\in G:\lambda(g)\leq K+1\}$

We now set $X_3$ to be the finite set $X_{G_3}=\bigcup_{j\in I_3} G_3\cdot x_j$. We again consider $X_0,X_1,X_2$ to be subsets of $X_3$. We have a metric $d_{X_3}=d_{G_3}$ defined using the weight function. What remains to check is that the canonical inclusions of $X_1$, resp. $X_2$ into $X_3$ are isometric. We shall do it for $X_1$, for $X_2$ it is analogous. So take some $g,h\in G_1$ and $i,j\in I_1$. We must check that $d_{X_3}(g\cdot x_i,h\cdot x_j)=d_{X_1}(g\cdot x_i,h\cdot x_j)$. Again, it is clear that $d_{X_3}(g\cdot x_i,h\cdot x_j)\leq d_{X_1}(g\cdot x_i,h\cdot x_j)$; suppose that there is a strict inequality. It follows that there is a path $e_1\ldots e_n$ from $g\cdot x_i$ to $h\cdot x_j$ such that $\sum_{l=1}^n w(e_l)<d_{X_1}(g\cdot x_i,h\cdot x_j)$. We claim that the length of the path $n$ is less or equal to $K$. Suppose that $n>K$. Then since for every $1\leq l\leq n$ we have $w(e_l)\geq m$, we get $$\sum_{l=1}^n w(e_l)\geq n\cdot m> K\cdot n\geq M.$$ However, by assumption $d_{X_1}(g\cdot x_i,h\cdot x_j)\leq M$, a contradiction.

Now, it follows that the path $e_1\ldots e_n$ lies within the finite set $\bigcup_{i\in I_3} \{g\in G_3:\rho(g)\leq K+1\}\cdot x_i$. Since $\phi$ is isometric with respect to $\lambda$ and $\rho$ on the ball $\{g\in G:\lambda(g)\leq K+1\}$ it follows that the path $e_1\ldots e_n$ from $G_3$ also exists in $X_{G'}$, and is, by definition, of the same length. However, we showed in the proof of Theorem \ref{thm_amal} that in $X_{G'}$ its weight was greater or equal to $d_{X_1}(g\cdot x_i,h\cdot x_j)$. This finishes the proof.
\end{proof}
\begin{remark}
We note that in the previous theorem it was essential that the actions were free. In that case, the residual finiteness of the free product of finite groups with amalgamation was sufficient. A strictly stronger notion than residual finiteness is the \emph{Ribes-Zalesski\v i property} (see \cite{RiZa}). Note that a one way how to formulate residual finiteness of $G$ is to say that the unit $1_G$ is closed in the profinite topology on $G$. In a similar spirit, one says that $G$ has the Ribes-Zalesski\v i property if for every tuple of finitely generated subgroups $H_1,\ldots,H_n\leq G$ their product $H_1\cdot\ldots\cdot H_n=\{h_1\cdot\ldots\cdot h_n: \forall i\leq n\; (h_i\in H_i)\}$ is closed in the profinite topology.

In \cite{Ro2}, Rosendal used the Ribes-Zalesski\v i property for finitary approximations of actions of groups on metric spaces by isometries. Similar ideas could be used to prove the amalgamation property for general, not necessarily free, actions if amalgamated free products had the Ribes-Zalesski\v i property. After proving the preceding theorem, it was pointed out to us by Julien Melleray that indeed amalgamated free products of two finite groups do have the Ribes-Zalesskii property. The proof follows the lines of Theorem 2 in \cite{Ba} and uses the fact that free groups have this property.
\end{remark}
Let $(G_n\curvearrowright X_n, (x_i)_{i\in I_n})_{n\in\Nat}$ be an enumeration of all pointed free actions of finite groups on finite metric spaces with rational distances. It follows from the previous theorem that it is a Fra\" iss\' e class. Indeed, it is clear from the proof that when working with rational spaces the amalgam will be rational as well. Moreover, the joint-embedding property is just a special case of the amalgamation property (note that any two actions have a common subaction, namely the action of a trivial group on a one-point space). So it has some Fra\" iss\' e limit $(\alpha_0: G\curvearrowright X, (x_i)_{i\in I})$, where $G$ is some countably infinite locally finite group, $X$ is a countably infinite rational metric space with countably infinite distinguished set of points $(x_i)_{i\in I}$ and $\alpha_0: G\curvearrowright X$ is a free action by isometries. 

It follows from the Fra\" iss\' e theorem that $(\alpha_0: G\curvearrowright X, (x_i)_{i\in I})$ has the following extension property:
\begin{fact}[The extension property]\label{extension_fact}
Let $F\leq G$ be a finite subgroup, $A\subseteq I$ a finite subset, and denote by $X_0$ the finite metric space $\bigcup_{i\in A} F\cdot x_i$. Consider the free pointed action $(F\curvearrowright X_0, (x_i)_{i\in A})$. Let $(H\curvearrowright Y,(y_j)_{j\in B})$ be some free pointed action of a finite group on a finite rational metric space and let $(\psi,\phi)$ is an embedding from $(F\curvearrowright X_0, (x_i)_{i\in A})$ to $(H\curvearrowright Y,(y_j)_{j\in B})$. Then there exists an embedding $(\bar{\psi},\bar{\phi})$ from $(H\curvearrowright Y,(y_j)_{j\in B})$ to $(G\curvearrowright X, (x_i)_{i\in I})$ such that $\bar{\psi}\circ \psi=\mathrm{id}_F$ and $\bar{\phi}\circ\phi=\mathrm{id}_{X_0}$.
\end{fact}

Now let $\mathbb{X}$ be the metric completion of $X$. The action $\alpha_0: G\curvearrowright X$ obviously extends to the action $\alpha: G\curvearrowright \mathbb{X}$ by isometries, which is no longer free though.

The following is a restatement of Theorem \ref{intro_mainthm} from Introduction using the just constructed action $\alpha: G\curvearrowright \mathbb{X}$.
\begin{thm}\label{main_thm}
The action $\alpha: G\curvearrowright \mathbb{X}$ is a universal action in the class of all actions of countable locally finite groups on all separable metric spaces by isometries.
\end{thm}
Before we prove the theorem we shall need few notions and lemmas.
\begin{defin}
Let $X$ be a set equipped with two pseudometrics $d$ and $p$. We define the distance $D(d,p)$ between these two pseudometrics as their supremum distance, i.e. $$D(d,p)=\sup_{x,y\in X} |d(x,y)-p(x,y)|.$$
\end{defin}
\begin{lem}\label{rational_approx_lem}
Let $(H\curvearrowright X,(x_i)_{i\in I})$ be a free pointed action by isometries of some finite group $H$ on a finite \underline{pseudometric} space $X=\bigcup_{i\in I} H\cdot x_i$ with pseudometric $d$. Then for any $\varepsilon>0$ there exists a rational \underline{metric} $p$ on $X$ such that the free action of $H$ on $(X,p)$ is still by isometries and $D(d,p)<\varepsilon$.
\end{lem}
\begin{proof}[Proof of Lemma \ref{rational_approx_lem}]
Enumerate by $(d_i)_{i\leq n}$ the distances from $(X,d)$ in an increasing order. Also, we may suppose that $\varepsilon<\min\{|k-l|:k\neq l,k,l\in\{d_i:i\leq n\}\cup\{0\}\}$.

For $i\leq n$, let $p_i$ be an arbitrary rational number from the open interval $(d_i+\frac{(n-i)\varepsilon}{n+1},d_i+\frac{(n+1-i)\varepsilon}{n+1})$. Now for a pair $x,y\in X$ set $p(x,y)=0$ if $x=y$ and for $x\neq y\in X$ set $$p(x,y)=p_i\text{ iff }d(x,y)=d_i.$$

Let us check that $p$ is a rational metric. By definition it is rational. It is clear that $p(x,y)=0$ iff $x=y$, and that it is symmetric, so we must just check the triangle inequality. Take a triple $x,y,z\in X$. We check that $p(x,z)\leq p(x,y)+p(y,z)$. If either $d(x,y)$ or $d(y,z)$ is bigger or equal to $d(x,z)$, then the same is true for $p(x,y)$, $p(y,z)$, $p(x,z)$ by definition. So we may suppose that $d(x,z)>\max\{d(x,y),d(y,z)\}$. Then by setting $d(x,z)=d_i$, $d(x,y)=d_j$ and $d(y,z)=d_k$, we have that $i>\max\{j,k\}$. We must check that $p_i\leq p_j+p_k$. However, we have $$p_i\leq d_i+\frac{(n+1-i)\varepsilon}{n+1}\leq d_j+\frac{(n-j)\varepsilon}{n+1}+d_k+\frac{(n-k)\varepsilon}{n+1}\leq p_j+p_k,$$ and we are done.
\end{proof}
\begin{lem}\label{glue_lem}
Let $H_1\leq H_2$ be two finite groups and $I\subseteq J$ two finite sets. Let $d$ be a metric on $X=\bigcup_{i\in I} H_1\cdot x_i$ and $p$ be a metric on $Y=\bigcup_{j\in J} H_2\cdot x_j\supseteq X$. Suppose that the canonical actions of $H_1$ on $(X,d)$, resp. of $H_2$ on $(Y,p)$ are by isometries. Suppose further that $D(d,p\upharpoonright X)\leq\varepsilon$. Then there exists a metric $\rho$ on $Z$, the disjoint union $X\subseteq\bigcup_{i\in I} H_2\cdot x_i\coprod \bigcup_{j\in J} H_2\cdot x_j=Y$ which is equal to $\bigcup_{i\in I} H_2\cdot x_i\cup \bigcup_{j\in J} H_2\cdot y_j$ such that
\begin{itemize}
\item $\rho$ extends both $d$ and $p$ on the corresponding subspaces,
\item for every $i\in I$, $\rho(x_i,y_i)\leq\varepsilon$,
\item the canonical action of $H_2$ on $Z$ is by isometries.

\end{itemize}
\end{lem}
\begin{proof}[Proof of Lemma \ref{glue_lem}]
As before, we define a weighted graph structure on $Z$. A pair $x,y$ is connected by an edge if and only if 
\begin{itemize}
\item either $x,y\in X$, resp. $x,y\in Y$, in such a case $w(x,y)=d(x,y)$, resp. $w(x,y)=p(x,y)$;
\item or there are $i\in I\subseteq J$ and $h\in H_2$ such that $x=h\cdot x_i$ and $y=h\cdot y_i$ or vice versa, in such a case we set $w(x,y)=\varepsilon$;
\item or $x=g\cdot x_i$, $y=h\cdot x_j$ such that $i\in I$ and $g^{-1}h\in H_1$; in such a case we set $w(x,y)=d(x_i,g^{-1}h x_j)$.

\end{itemize}
 
It is again immediate that the graph is connected, thus it determines a metric $\rho$ on $Z$, and the canonical action of $H_2$ on $Z$ is by isometries. We need to check that $\rho$ extends $d$ and $p$. We check both simultaneously.

Fix $x,y$ such that either $x,y\in X$ or $x,y\in Y$. Suppose that $\rho(x,y)<d(x,y)$ (it is again clear that $\rho(x,y)\leq d(x,y)$), resp. $\rho(x,y)<p(x,y)$ depending on where $x,y$ lie. Then there is an edge path $e_1\ldots e_n$ such that $w(e_1)+\ldots+w(e_n)<d(x,y)$, resp. $w(e_1)+\ldots+w(e_n)<p(x,y)$. We shall again prove the claim by induction on the length of the edge path. The case $n=1$ is clear. Suppose we have proved it for all $l<n$ and all edge paths of length at most $l$ between all pairs $x,y\in X$ and all pairs $x,y\in Y$. We may suppose that there are not two neighboring edges $e_i$ and $e_{i+1}$ such that both of them lie in $X$ or both of them lie in $Y$, for otherwise we could contract them into a single edge using triangle inequality in $X$, resp. $Y$.

Suppose first that $x,y\in X$ and let $x=g\cdot x_i$ and $y=h\cdot x_j$, for some $g,h\in H_1$ and $i,j\in I$. Denote by $\bar{X}$ the set $\bigcup_{i\in I} H_2\cdot x_i\supseteq X$. Notice that there is no edge between an element $z\in X$ and an element $z'\in \bar{X}\setminus X$. Thus we may suppose that $e_1$ is an edge between $x=g\cdot x_i$ and $g\cdot y_i$ and $e_n$ is an edge between $h\cdot y_j$ and $h\cdot x_j=y$. Indeed, otherwise either $e_1$ is an edge within $X$, so we may use the inductive assumption for the subpath $e_2,\ldots,e_n$, or $e_n$ is an edge within $X$ and we may use the inductive assumption for the subpath $e_1,\ldots,e_{n-1}$. It follows that $e_2,\ldots,e_{n-1}$ is an edge path of length strictly less than $n$ between two elements of $Y$, thus by inductive assumption we may suppose that $n=3$ and $e_2$ is an edge between $g\cdot y_i$ and $h\cdot y_j$ and we have $$w(g\cdot y_i,h\cdot y_j)=p(g\cdot y_i,h\cdot y_j)\geq d(g\cdot x_i,h\cdot x_j)-\varepsilon.$$ However, since $w(e_1)=w(e_3)=\varepsilon$, we get that $$d(x,y)=d(g\cdot x_i,h\cdot x_j)<w(e_1)+w(e_2)+w(e_3),$$ a contradiction.

Suppose now that $x,y\in Y$ and again let $x=g\cdot y_i$ and $y=h\cdot y_j$, for some $g,h\in H_2$ and $i,j\in J$. As in the paragraph above, we may without loss of generality assume that $e_1$ is an edge between $x=g\cdot y_i$ and $g\cdot x_i$ and $e_n$ is an edge between $h\cdot x_j$ and $h\cdot y_j=y$; thus in particular $i,j\in I$. If both $g,h\in H_1$ then $g\cdot x_i,h\cdot x_j\in X$ and we are done by the same argument as in the paragraph above. So suppose that at least one of $g,h$ is in $H_2\setminus H_1$. Say $g\in H_2\setminus H_1$, i.e. $g\cdot x_i\in \bar{X}\setminus X$. Since there is no edge between an element from $X$ and an element from $\bar{X}\setminus X$ there exists a minimal $l\leq n$ such that $e_2,\ldots,e_{l-1}$ is a path within $\bar{X}\setminus X$ and $e_l$ is an edge between an element from $\bar{X}\setminus X$ and an element from $Y$. If $l<n$ then we use the inductive hypothesis, so suppose that $l=n$, i.e. the subpath $e_2,\ldots,e_{n-1}$ is within $\bar{X}\setminus X$. Note also that there is an edge between elements $f\cdot x_k$ and $f'\cdot x_{k'}$ in $\bar{X}\setminus X$ if and only if $f^{-1}f'\in H_1$. It follows that $g^{-1}h\in H_1$. Translating the whole path $e_1,\ldots,e_n$ by $g^{-1}$ we does not change the distance ($g^{-1}$ acts as an isometry). Thus we may assume that $g=1$ and $h\in H_1$. However, then we are again done by an argument used above.
\end{proof}
\begin{proof}[Proof of Theorem \ref{main_thm}]
Let $H\curvearrowright Z$ be an action of an infinite locally finite group by isometries on a separable metric space. It is sufficient to prove the theorem in case $Z$ is countable. Indeed, in the general case we would find a countable dense $H$-invariant subspace $Z'$. Then the action on the metric completion of $Z'$, thus in particular on $Z$, is uniquely determined by its behavior on $Z'$. Since the space $\mathbb{X}$ is complete, we are done.

So assume that both $H$ and $Z$ are countable. Without loss of generality we assume that $Z$ has infinitely many $H$-orbits and let $(z_n)_{n\in\Nat}$ be a sequence which picks one single element from each orbit, i.e. we may write the metric space $Z$ as $\bigcup_{n\in\Nat} H\cdot z_n$ with a pseudometric $d$.   Also, without loss of generality we shall assume that $H$ is infinite and write $H$ as $H_1\leq H_2\leq H_3\leq \ldots$ which is an increasing chain of finite subgroups of $H$ whose union is $H$. Moreover, for every $n$ define $Z_n$ to be the finite pseudometric subspace $\bigcup_{i\leq n} H_n\cdot z_i\subseteq Z$.

For every $n$ consider the free pointed action $(H_n\curvearrowright Z_n, (z_i)_{i\leq n})$. By Lemma \ref{rational_approx_lem} there exists a rational metric $p_n$ on $Z_n$ such that $D(d\upharpoonright Z_n,p_n)<1/2^{n+1}$. In particular, we get a free action of $H_n$ on $(Z_n,p_n)$ by isometries. It follows that for every $n$ we have $D(p_n,p_{n+1}\upharpoonright Z_n)<1/2^n$. By Lemma \ref{glue_lem}, for every $n$ we can define a rational metric $\rho_n$ on a disjoint union of $Z_n\coprod Z_{n+1}=\left(\bigcup_{i\leq n} H_n\cdot z_i\right)\cup \left(\bigcup_{j\leq n+1} H_{n+1}\cdot z'_j\right)$ which is free and by isometries, and which extends the original metrics and for $i\leq n$ we have $\rho_n(z_i,z'_i)=1/2^n$. Now by a successive application of Fact \ref{extension_fact}, the extension property of $(G\curvearrowright X,(x_{i\in I})$, we obtain
\begin{itemize}
\item an increasing chain of finite subgroups $H'_1\leq H'_2\leq\ldots\leq G$ and isomorphisms $\psi_i:H'_i\rightarrow H_i$, for $i\in \Nat$, such that $\psi_i\subseteq \psi_{i+1}$ for every $i$. Thus $\psi=\bigcup_i \psi$ is an isomorphism between $H'=\bigcup_n H'_n$ and $H$;
\item isometric embeddings $\phi_n:Z_n\coprod Z_{n+1}\hookrightarrow X$ such that $\phi_n\upharpoonright Z_{n+1}=\phi_{n+1}\upharpoonright Z_{n+1}$, for every $n$;
\item for every $n$, we have that the free actions $H_n\curvearrowright Z_n$ and $H'_n\curvearrowright \phi_n[Z_n]$ are isometric.

\end{itemize}
For every $i$ we have that the sequence $(\phi_n(z_i))_{n\geq i}$ is Cauchy, since $d_X(\phi_n(z_i),\phi_{n+1}(z_i))=1/2^n$. Let $y_i\in \mathbb{X}$ be the limit of that sequence. Consider the subset $Z'=\bigcup_{i\in\Nat} H'\cdot y_i\subseteq \mathbb{X}$. It follows it is naturally isometric to $Z$. Indeed, take any $x,y\in Z$ and write them as $x=h\cdot z_i$ and $y=g\cdot z_j$ for some $h,g\in H$ and $i,j\in\Nat$. Since $H$ and $H'\leq G$ are isomorphic, let $h',g'$ be the corresponding elements of $H'\leq G$ and consider the elements $h'\cdot y_i, g'\cdot y_j\in Z'\subseteq \mathbb{X}$. Then $$d(h'\cdot y_i,g'\cdot y_j)=\lim_n d(h'\cdot \phi_n(z_i),g'\cdot \phi_n(z_j))=\lim_n d_Z(h\cdot z_i,g\cdot z_j)+o(n),$$ where $o(n)\in [0,1/2^n]$, so the claim is proved.

Finally, consider the restriction of the action $G\curvearrowright \mathbb{X}$ on $H'\curvearrowright \bigcup_{i\in\Nat} H'\cdot y_i$. It follows from the approximation above that it is isometric to the action $H\curvearrowright Z$, and we are done.
\end{proof}
Finally, we show that the group $G$ is isomorphic to the Hall's universal locally finite group and that the space $X$ is isometric to the rational Urysohn space, so the completion $\mathbb{X}$ is isometric to the Urysohn universal space. It is just the use of the extension property of $G\curvearrowright X, (x_i)_{i\in I}$ from Fact \ref{extension_fact}. These are standard arguments, so we omit some details.

For the former it is necessary to show that $G$ has the extension property. That is, whenever $F\leq G$ is some finite subgroup and $H\geq F$ is some abstract finite supergroup of $F$, i.e. a supergroup of $F$ that does not in principle lie in $G$, then we can actually find a copy $H'$ of $H$ within $G$ so that it is a supergroup of $F$ there, i.e. $F\leq H'\leq G$.

So pick some finite subgroup $F\leq G$ and some abstract supergroup $H\geq F$. The Fra\" iss\' e limit $G\curvearrowright X, (x_i)_{i\in I}$ is a direct limit of a sequence of some finite actions  $(G_n\curvearrowright X_n, (x_i)_{i\in I_n})_{n\in\Nat}$. Take $n$ so that $F\leq G_n$ and consider the subaction $F\curvearrowright X'_n,(x_i)_{i\in I_n}$, where $X'_n=\bigcup_{i\in I_n} F\cdot x_i$. It is possible to use Lemma \ref{glue_lem} to extend this action to an action of $H$ on $\bigcup_{i\in I_n} H\cdot x_i$. Then we use the extension property of $G\curvearrowright X, (x_i)_{i\in I}$ to find the action $H\curvearrowright \bigcup_{i\in I_n} H\cdot x_i$ within the universal one, thus in particular to find a copy of $H$ within $G$ that is a supergroup of $F$.

Now for the latter, it is necessary to show that the countable rational metric space $X$ has the extension property. That is, whenever $A\subseteq X$ is some finite subspace and $A\subseteq B$ is finite abstract extension, still a rational metric space, then we can actually find this extension within $X$. So take some finite $A\subseteq X$. As above, find some $n$ so that $A\subseteq X_n$. By extending the metric by metric amalgamation if necessary we may assume that $A=X_n$. Set $I'_n=I\cup (B\setminus X_n)$ and $X'_n=\bigcup_{i\in I'_n} G_n\cdot x_i$. Clearly, $A=X_n\subseteq B\subseteq X'_n$. By using the technique with defining a weighted graph structure on $X'_n$ we can extend the metric from $B$ to $X'_n$ so that $G_n$ acts on $X'_n$ by isometries. Then we use the extension property of $G\curvearrowright X, (x_i)_{i\in I}$ to get a copy of $X'_n$, thus also of $B$, in $X$ so that it is an extension of $A$ there.\\

\subsection{Amalgamation property for actions of abelian groups}
A possible modification of the ideas above would be to consider the class of (pointed free) actions of abelian groups on finite metric spaces. One might expect that it has the amalgamation property as well and the (completion of the) limit will be a universal action of $\bigoplus_{n\in\Nat} \Rat/\Int$ on the Urysohn space $\Ur$, where $\bigoplus_{n\in\Nat} \Rat/\Int$ is easily checked to be the Fra\" iss\' e limit of the class of all finite abelian groups. We show that it is not the case. More precisely, we prove the following proposition.
\begin{prop}
The class of all pointed free actions of finite abelian groups on finite metric space does not have the amalgamation property.
\end{prop}
\begin{proof}
In the proof we stick to the same notation we used for actions of general non-abelian groups. Let us consider actions $(G_i\curvearrowright X_i, (x_1,x_2,x_3))$, for $i=0,1,2$, where $G_0$ is the trivial group and $G_1=G_2=\Int/2\Int$. We denote the single non-zero element of $G_1$ as $g$ and the single non-zero element of $G_2$ as $h$. The embeddings of $(G_0\curvearrowright X_0, (x_1,x_2,x_3))$ into $(G_1\curvearrowright X_1, (x_1,x_2,x_3))$, resp. $(G_2\curvearrowright X_2, (x_1,x_2,x_3))$ are obvious. We set $d(x_1,x_2)=d(x_2,x_3)=10$ and $d(x_1,x_3)=16$. On $X_1=\bigcup_{i\leq 3} G_1\cdot x_i$, we then have $d(g\cdot x_1,g\cdot x_2)=d(g\cdot x_2,g\cdot x_3)=10$ and $d(g\cdot x_1,g\cdot x_3)=18$. Additionally, we set $d(x_1,g\cdot x_2)=d(x_2,g\cdot x_3)=d(g\cdot x_1,x_2)=d(g\cdot x_2,x_3)=8$. This gives $X_1$ a structure of a connected weighted graph on which $G_1$ acts by preserving the weighted graph structure. So we can define a metric on $X_1$ as the graph metric. It is easy to check that this metric extends $d$ and $G_1$ acts on $X_1$ by isometries.

On $X_2=\bigcup_{i\leq 3} G_2\cdot x_i$, we also have $d(h\cdot x_1,h\cdot x_2)=d(h\cdot x_2,h\cdot x_3)=10$ and $d(h\cdot x_1,h\cdot x_3)=18$. Additionally, we set $d(x_2, h\cdot x_2)=1$, $d(x_1,h\cdot x_3)=d(h\cdot x_1, x_3)=18$. This again gives $X_2$ a structure of a connected weighted graph on which $G_2$ acts by preserving the weighted graph structure. So, as for $X_1$, we define a metric on $X_2$ as the graph metric, which clearly extends $d$ and $G_2$ acts on $X_2$ by isometries.

We claim that these two actions cannot be amalgamated over $G_0\curvearrowright X_0$. Suppose that there is an amalgam action of some $G_3$ on $X_3=\bigcup_{i\leq 3} G_3\cdot x_i$. We denote the metric on $X_3$ again just by $d$. Moreover, we denote the image of $g\in G_1$, resp. of $h\in G_2$, in $G_3$ again by $g$, resp. by $h$. Then we have $$18=d(x_1,h\cdot x_3)\leq d(x_1,g\cdot x_2)+d(g\cdot x_2,(h+g)\cdot x_2)+d((h+g)\cdot x_2,h\cdot x_3)=8+1+8=17,$$ a contradiction.
\end{proof}
However, we do not know if this class has the \emph{cofinal} amalgamation property. That is, whether there exists a proper subclass $\mathcal{C}$ having the amalgamation property such that every pointed free action of an abelian group on a finite metric space embeds to an action from $\mathcal{C}$. The existence of such a class would also imply the existence of a universal action of $\bigoplus_{n\in\Nat} \Rat/\Int$ on the Urysohn space by isometries. So in particular the following is left open.
\begin{question}
Does there exist an action of a countable torsion abelian group on a separable metric space by isometries which is universal for the class of all actions of torsion abelian groups on all separable metric spaces?
\end{question}
\subsection{Non-universality results}
In this subsection we justify why we work with Definition \ref{def_univaction} rather than with Definition \ref{def_univaction2}, where a different notion of universality was defined.
\begin{thm}\label{no_univaction}
Let $G$ be a countably infinite group. Then there is no universal action of $G$ in both the class of all actions of $G$ on all separable metric spaces and the class of all actions of $G$ on all separable Banach spaces.
\end{thm}
\begin{proof}
Suppose the contrary, first for the metric spaces. That is, suppose there exists a separable metric space $X$ and an action $\alpha: G\curvearrowright X$ by isometries such that for any action $\beta:G\curvearrowright Y$ by isometries on a separable metric space $Y$ there is a $G$-equivariant isometric embedding of $Y$ into $X$.

Since $G$ is countably infinite we can find a sequence $(g_n)_n$ that generates $G$ and moreover, for any $n\neq m$ we have $g_n\notin \{g_m,g_m^{-1}\}$.

For any $x\in 2^\Nat$ let $\lambda'_x:\{g_n,g^{-1}_n:n\in\Nat\}\rightarrow \Nat$ be defined as follows: $$\lambda'_x(g)=\begin{cases}
1 & g\in\{g_n,g^{-1}_n\}\wedge x(n)=0,\\
2 & g\in\{g_n,g^{-1}_n\}\wedge x(n)=1.
\end{cases}$$

Finally, for any $x\in 2^\Nat$ define a length function $\lambda_x:G\rightarrow\Nat$ as follows: for any $g\in G$, set $$\lambda_x(g)=\min\{\sum_{i=1}^m \lambda'_x(h_i): g=h_1\ldots h_m, (h_i)_{i=1}^m\subseteq \{g_n,g^{-1}_n:n\in\Nat\}\}.$$
We claim that $\lambda_x$ extends $\lambda'_x$, i.e. for every $g\in \{g_n,g^{-1}_n:n\in\Nat\}$, $\lambda_x(g)=\lambda'_x(g)$. It suffices to show that for any $n$ such that $x(n)=1$ we have $\lambda_x(g_n)=2$. Suppose the contrary. Then necessarily $\lambda_x(g_n)=1$, so by definition $g_n=g_m$ or $g_n=g^{-1}_m$ for $m$ such that $x(m)=0$. However, that contradicts our assumption .\\

Now for every $x\in 2^\Nat$ take the left-invariant metric $d_x$ on $G$ induced by $\lambda_x$. The action of $G$ on itself by left translations is then an action of $G$ on $(G,d_x)$ by isometries. We claim that there is $x\in 2^\Nat$ such that there is no $G$-equivariant isometric embedding of $(G,d_x)$ into $X$. Suppose otherwise that for every $x\in 2^\Nat$ there is a $G$-equivariant isometric embedding $\iota_x$ of $(G,d_x)$ into $X$. For every $x\in 2^\Nat$ denote $\iota_x(1_G)\in X$ by $z_x$. Then for $x\neq y\in 2^\Nat$ we have $d_X(z_x,z_y)\geq 1/2$, for if $d_X(z_x,z_y)<1/2$ and $n\in\Nat$ is such that $x(n)\neq y(n)$, say $x(n)=1$, $y(n)=0$, then $$1/2>d_X(z_x,z_y)=d_X(g_n\cdot z_x,g_n\cdot z_y)\geq$$ $$|d_X(g_n\cdot z_x,z_x)-d_X(z_x,z_y)-d_X(z_y,g_n\cdot z_y)|> 1/2,$$ a contradiction. Thus we get that $\{z_x:x\in 2^\Nat\}\subseteq X$ is a $1/2$-separated uncountable set in $X$ which contradicts the separability of $X$.\\

To prove the same for the category of Banach spaces, we can for example extend the action of $G$ on $(G,d_x)$, for every $x\in 2^\Nat$, to an action of $G$ on the Lipschitz-free Banach space $F(G,d_x)$ over $(G,d_x)$ (see \cite{GK} and \cite{We} for information about Lipschitz-free Banach spaces). That is, consider a real vector space $V_G$ with $G\setminus\{1_G\}$ as the free basis, and $1_G$ as a zero. Define a norm $\|\cdot\|_x$ on $V_G$ as follows: for $v=\alpha_1 g_1+\ldots+\alpha_n g_n$ set $$\|v\|_x=\min\{\sum_{i=1}^m |\beta_i|\cdot d_x(h_i,h'_i):v=\sum_{i=1}^m \beta_i(h_i-h'_i)\}.$$  Then it is easy to check (and it is a standard fact about Lipschitz-free spaces) that for any $g,h\in G$, $\|g-h\|_x=d_x(g,h)$. $G$ acts by (affine) isometries on $(V_G,\|\cdot\|_x)$ in the following way: for $h\in G$ and $\alpha_1 g_1+\ldots+\alpha_n g_n\in V_G$ we set $h\cdot (\alpha_1 g_1+\ldots+\alpha_n g_n)=(\alpha_1 hg_1+\ldots+\alpha_n hg_n)-(\alpha_1+\ldots+\alpha_n-1)h$. It is easy to check that this gives an action of $G$ on $(V_G,\|\cdot\|_x)$ by isometries which extends the action of $G$ on itself by translation. It also extends to an action of $G$ on the completion $W_x$. Then arguing exactly the same as with the metric space one can show that it is not possible to embed in a $G$-equivariant way all the spaces $W_x$,$x\in 2^\Nat$, into a single separable Banach space with an action of $G$.
\end{proof}

Second, we explain that the universality cannot be naturally extended beyond the class of locally finite groups - at least in the case when we do not restrict the class of separable metric spaces.

\begin{prop}
Let $G$ be a countably infinite non-locally finite group. Let $\Class$ be the class of all actions of groups isomorphic to subgroups of $G$ on all separable metric spaces. Then $\Class$ does not admit a universal action.
\end{prop}
\begin{proof}
Suppose there is such an action $\alpha: F\curvearrowright X$, where $F$ is some subgroup of $G$ and $X$ is some separable metric space. Since $G$ is not locally finite it contains a finitely generated infinite subgroup $H\leq G$. By the proof of Theorem \ref{no_univaction} there are continuum many somewhat different left-invariant metrics $(d_x)_{x\in 2^\Nat}$ on $H$. Note that $F$ contains at most countably many subgroups isomorphic to $H$. By the pigeonhole principle there is one fixed subgroup $H'\leq F$ isomorphic to $H$ and an uncountable subset $I\subseteq 2^\Nat$ such that for each $x\in I$ there is an $H'$-equivariant isometric embedding of $(H',d_x)$ into $X$. We reach a contradiction with separability of $X$ by the same argument as in the proof of Theorem \ref{no_univaction}.
\end{proof}
We have not considered any classes of actions where the class of separable metric spaces is restricted.
\begin{problem}
Find a natural class of separable metric spaces for which there are universal actions of non-locally finite groups.
\end{problem}
\section{Genericity of the action}
The objects constructed using the Fra\" iss\' e theory enjoy two interesting properties, the universality and the homogeneity. In the previous section we focused solely on the universality of the constructed action as that was the more interesting part in our point of view. We shall not explore the homogeneity here, however we want to focus on a property which often accompanies homogeneity, in fact is a homogeneity in disguise in a sense. That is the `genericity'.

Let us start with a general discussion. Fix some countable group $G$. Let $X$ be some Polish metric space, i.e. a complete separable metric space. We want to define a space of all actions of $G$ on $X$ by isometries.

We have that $\Iso(X)$, the group of all isometries on $X$ with the pointwise-convergence, or equivalently compact-open, topology is a Polish group, i.e. a completely metrizable second-countable topological group. Fixing a countable dense subset $\{x_i:i\in\Nat\}\subseteq X$ we may define a compatible complete metric $\rho$ on $\Iso(X)$ as follows: for $\phi,\psi\in\Iso(X)$ we set $$\rho(\phi,\psi)=\sum_{i=1}^\infty \frac{\min\{d_X(\phi(x_i),\psi(x_i)),1\}}{2^i}.$$
Since every action $\alpha:G\curvearrowright X$ by isometries is in unique correspondence with some homomorphism $f:G\rightarrow \Iso(X)$, we may define the space $\Act_G(X)$ of all actions of $G$ on $X$ by isometries as the space of all homomorphisms of $G$ into $\Iso(X)$. $\Act_G(X)$ is a closed subspace of the product space $\Iso(X)^G$, thus a Polish space.\\

There has been a recent research on investigating which countable groups admit generic homomorphisms into certain Polish groups. That means, fix a countable group $G$ and a Polish group $H$. Denote by $\Hom(G,H)$ the Polish space of all homomorphisms of $G$ into $H$. Note that there is a natural equivalence relation on the space $\Hom(G,H)$, that is of conjugation, where two homomorphisms $f,g:G\rightarrow H$ are conjugate if there exists an element $x\in H$ such that $f=x^{-1}g x$. Say that a homorphism $f$ is \emph{generic} if it has a comeager conjugacy class. In \cite{DoMaVa}, the author with Malicki and Valette prove that a countable group $G$ with property (T) and such that finite-dimensional representations are dense in the unitary dual $\hat G$ has a generic unitary representation. That is, a generic homomorphism into the unitary group of a separable infinite-dimensional Hilbert space equipped with the strong operator topology. This was implicitly present already in \cite{KLP} (see Theorem 2.5), from which one can also derive the converse. The existence of such a countably infinite group seems to be open though. In \cite{Ro} Rosendal proved that every finitely generated group with the Ribes-Zalesski\v i property (\cite{RiZa}), i.e. products of finitely generated subgroups are closed in the profinite topology of the group, has a generic action on the rational Urysohn space. More recently, Glasner, Kitroser and Melleray (\cite{GKM}) characterized those countable groups that have generic permutation representations, i.e. generic homomorphisms into $S_\infty$, the full permutation group of the natural numbers. See also the results about generic representations in other metric structures \cite{DoMa}.\\

We shall prove a genericity result for the universal action from Theorem \ref{main_thm}. However, it turns out that the standard equivalence relation on the space of actions is too strong. Indeed, by Melleray in \cite{Me} (also independently proved in \cite{DoMa}), the conjugacy class of every action of a countably infinite group on the Urysohn space is meager, so not generic. Thus we shall weaken the equivalence relation of being conjugate by also allowing group automorphisms. Let us state that precisely in the following definition.
\begin{defin}
Let $G$ be a countable group and $X$ a Polish metric space. Say that two homomorphisms $f,g:G\rightarrow \Iso(X)$ are \emph{weakly equivalent} if there exist an autoisometry $\phi:X\rightarrow X$ and an automorphism $\psi:G\rightarrow G$ such that for all $x\in X$ and $v\in G$ we have $$f(v)x=\phi^{-1}g(\psi(v))\phi x.$$

Moreover, we say that an element $f\in\Act_G(X)$ is \emph{weakly generic} if it has a comeager equivalence class in the weak equivalence.
\end{defin}
We shall prove the following.
\begin{thm}\label{thm_generic}
Let $G$ be the Hall's universal locally finite group. The universal action $\alpha: G\curvearrowright \Ur$ from Theorem \ref{main_thm} is weakly generic.
\end{thm}
Theorem \ref{thm_generic} has an interesting corollary which we state after the proof the theorem.\\

We need some notions. When $F$ and $F'$ are two isomorphic finite groups, $I$ and $I'$ two finite bijective sets and $d$, resp. $d'$ a metric on $\bigcup_{i\in I} F\cdot x_i$, resp. on $\bigcup_{i\in I'} F'\cdot y_i$, we denote by $D((F,\{x_i:i\in I\},d),(F',\{y_i:i\in I'\},d'))$, analogously as in the previous section, the supremum distance $\sup_{i,j\in I,g,h\in F} |d(g\cdot x_i,h\cdot x_j)-d'(g'\cdot y_{i'},h'\cdot y_{j'})|$, where $g',h'\in F'$ are the images of $g,h\in F$ under the given isomorphism between $F$ and $F'$ and $i',j'\in I'$ are the images of $i,j\in I$ under the given bijection between $I$ and $I'$. Such an isomorphism and a bijection will be never explicitly mentioned, it should be always clear from the context. Also, we shall often write $D((F,\{x_i:i\in I\}),(F',\{y_i:i\in I'\}))$, thus suppressing the metrics from the notation; they should also be clear from the context. The following fact follows from Lemma \ref{glue_lem}, however we will state it here since it will be used extensively.
\begin{fact}\label{glue_fact}
Suppose we are given two finite isomorphic groups $F$ and $F'$, finite bijective sets $I$ and $I'$, and metrics $d$ and $d'$ on $\bigcup_{i\in I} F\cdot x_i$, resp. on $\bigcup_{i\in I'} F'\cdot y_i$. Suppose moreover that $D((F, \{x_i:i\in I\}),(F',\{y_i:i\in I'\}))<\varepsilon$ for some $\varepsilon>0$. Then there exists a metric $\rho$ on $\bigcup_{i\in I\cup I'} F\cdot x_i$ such that
\begin{itemize}
\item $D((F,\{x_i:i\in I\},d),(F,\{x_i:i\in I\},\rho))=0$, i.e. $\rho$ extends $d$;
\item $D((F,\{x_i:i\in I'\},\rho),(F',\{y_i:i\in I'\},d'))=0$;
\item for every $i\in I$ we have $\rho(x_i,x_{i'})\leq \varepsilon$.

\end{itemize} 
\end{fact}
\begin{remark}\label{remark_converse}
Conversely, suppose that a finite group $F$ acts freely on some metric space $Y$ and let $\{y_i:i\in I\}$ and $\{z_i:i\in I\}$ be two finite subsets of $Y$ indexed by the same set such that for every $i\in I$, $d_Y(y_i,z_i)<\varepsilon$. Then $D((F,\{y_i:i\in I\}),(F,\{z_i:i\in I\}))<2\varepsilon$.
\end{remark}

We shall now define a subset of $\Act_G(\Ur)$. By $\Rat\Ur$ we denote the rational Urysohn space, a countable dense subset of $\Ur$. We shall denote the pointed free rational actions of finite groups by $(F,\{x_i:\in I\})$ and we write $(F,\{x_i:i\in I\})\leq (H,\{x_i:i\in I'\})$ to denote that the former actions embeds into the latter. To simplify the notation, we always assume in such a case that $F\leq H$ and $I\subseteq I'$. Recall that the class $\Age$ of all pointed free rational actions by finite groups is countable.

By $\mathbb{D}$ we denote the subset of $\Act_G(\Ur)$ of all actions $G\curvearrowright \Ur$ satisfying:\\

for all $\varepsilon>\varepsilon'>0$, for all $(F,\{x_i:i\in I\})\leq (H,\{x_i:i\in I'\})\in\Age$ and for every subgroup $F'\leq G$ isomorphic to $F$ and all $\{u_i:i\in I\}\subseteq \Rat\Ur$ such that
 $$D((F,\{x_i:i\in I\}),(F',\{u_i:i\in I\})<\varepsilon'$$ there exist a subgroup $F'\leq H'\leq G$ isomorphic to $H$, and points $\{u_i:i\in I'\}\subseteq \Rat\Ur$ such that $$D((H,\{x_i:i\in I'\}),(H',\{u_i:i\in I'\}))<\varepsilon.$$

We shall refer to the property above as $D$-property. A simple computation shows that the $D$-property is a $G_\delta$ condition, i.e. $\mathbb{D}$ is a $G_\delta$ set. It is non-empty since the universal action $\alpha: G\curvearrowright \Ur$ from Theorem \ref{main_thm} clearly belongs to $\mathbb{D}$. We claim that its weak equivalence class is dense. Indeed, fix an open neighborhood of some action $\beta: G\curvearrowright \Ur$ which is given by finitely many group elements, finitely many elements from $\Ur$ and some $\varepsilon>0$. Without loss of generality, we may assume that these group elements form a finite subgroup $F\leq G$ and these finitely many elements from $\Ur$ form a finite subset $A\subseteq\Ur$ invariant under the subaction of this finite group. We may view this subaction as a pointed free action on a finite pseudometric space $(Y,p)$. Let $\tau:Y\rightarrow A$ be the corresponding surjection (which is not injective unless the action of $F$ on $A$ is free). By Lemma \ref{rational_approx_lem}, we may approximate this pseudometric by a rational metric $r$ such that the action is still by isometries and $D((F,Y,p),(F,Y,r))<\varepsilon/2$. This finite action $F\curvearrowright Y$ (or its equivalence class) on a rational metric space belongs to the Fra\" iss\' e class of all pointed free actions of finite groups on finite rational metric spaces. Therefore, there is a subaction $F'\curvearrowright X_n$ of the universal action $\alpha$ that is equivalent to $F\curvearrowright Y$. That is, $F$ is isomorphic to $F'$, and between $Y$ and $X_n$ there is an equivariant isometry. Moreover, by the density of $\Rat\Ur$ in $\Ur$ and the extension property of $\Rat\Ur$ we may realize $Y$ as a subset of $\Rat\Ur\subseteq \Ur$ so that for every $y\in Y$ we have $d_\Ur(y,\tau(y))<\varepsilon/2$, where this last inequality is possible because of our assumption $D((F,Y,p),(F,Y,r))<\varepsilon/2$. Since $Y$ and $X_n$ are isometric, by the homogeneity of $\Ur$ there exists an autoisometry $\psi:\Ur\rightarrow\Ur$ such that $\psi(X_n)=Y$. Conjugating the action $\alpha$ with $\psi$ gives us an action $\alpha'$ where $F'$ acts on $Y$ in the same way as $F$ acts on $Y$. Since $F'$ and $F$ are isomorphic, by the homogeneity of the Hall's group $G$ there exists an automorphism $\phi:G\rightarrow G$ with $\phi(F)=F'$. `Shifting' the action $\alpha'$ by $\phi$, gives us a weakly equivalent action $\alpha''$, i.e. $\alpha''(g,x)=\alpha'(\phi(g),x)$. The action $\alpha''$ is then easily checked to be in the given neighborhood of $\beta$.

Now we need to show that any two actions from $\mathbb{D}$ are weakly equivalent. That is, for actions $\alpha,\beta\in\mathbb{D}$ we need to find an automorphism $\phi:G\rightarrow G$ and an autoisometry of $\psi:\Ur\rightarrow\Ur$ such that for all $g\in G$ and $x\in\Ur$ we have $$\alpha(g,x)=\beta(\phi(g),\psi(x)).$$

We now fix two actions $\alpha,\beta\in\mathbb{D}$ and show that. Let $(z_n)_{n\in\Nat}$ be some enumeration of $\Rat\Ur$ such that for each $i_0\in\Nat$ both sets $\{z_i:i\geq i_0,i\text{ is odd}\}$ and $\{z_i:i\geq i_0,i\text{ is even}\}$ are dense in $\Ur$. Also, write $G$ as an increasing union $G_1\leq G_2\leq G_3\leq\ldots$ of finite subgroups.

By induction, we shall find for each $n\in\Nat$:
\begin{itemize}
\item an increasing sequence of finite groups $H_1\leq\ldots\leq H_n\leq G$ and $H'_1\leq\ldots\leq H'_n$ such that for each $i\leq n$, $H_i$ and $H'_i$ are isomorphic by some $\phi_i$ and $\phi_i\supseteq \phi_{i-1}$, and for every odd $i\leq n$ we have that $G_i\leq H_i$, and for every even $i\leq n$ we have that $G_i\leq H'_i$;
\item for each $i\leq n$, sequences $(u_i^j)_{j=i}^n\subseteq \Rat\Ur$ and $(v_i^j)_{j=i}^n\subseteq \Rat\Ur$ such that
\begin{itemize}
\item for each $i\leq n$ and $i\leq j<k\leq n$, $d(u_i^j,u_i^k)\leq 1/2^{j+1}$ and  $d(v_i^j,v_i^k)\leq 1/2^{j+1}$,
\item for every odd $i\leq n$, $u_i^i=z_i$, and for every even $i\leq n$, $v_i^i=z_i$;
\end{itemize}
\item $D((H_n,\{u_i^n:i\leq n\}),(H'_n,\{v_i^n:i\leq n\}))<1/2^{n+1}$.

\end{itemize}

Once the induction is finished, we have that $G=\bigcup_n H_n=\bigcup_n H'_n$, i.e $\phi=\bigcup_n \phi_n:G\rightarrow G$ is an isomorphism.  Also we have that for every $i\leq n$ the sequences $(u_i^n)_n$ and $(v_i^n)_n$ are Cauchy in $\Rat\Ur$, thus they have some limit $u_i\in\Ur$, resp. $v_i\in\Ur$. It follows from the inductive assumption that both $\{u_i:i\in\Nat\}$ and $\{v_i:i\in\Nat\}$ are dense in $\Ur$ and that the map sending $u_i$ to $v_i$ is an isometry which extends to an autoisometry $\psi$ of $\Ur$. By the limit argument we get that the actions $\alpha$ and $\beta$ are weakly equivalent witnessed by $\phi$ and $\psi$. Thus we need to describe the inductive steps to finish the proof.\\

\noindent {\bf The first and second step of the induction.} Set $H_1=G_1$, $u_1^1=z_1$. By Lemma \ref{rational_approx_lem} there exists $(H_1, \{x_1\})\in\Age$ such that $D((H_1,\{u_1^1\}),(H_1,\{x_1\})<1/2$. Since $\beta$ satisfies the $D$-property, there is a subgroup $H'_1\leq G$ isomorphic to $H_1$ (via some $\phi_1$) and some $v_1^1$ which, because of Remark \ref{remark_converse} we may find in $\Rat\Ur$, such that\\ $D((H_1,\{x_1\}),(H'_1,\{v_1^1\})<1/2$, thus by triangle inequality\\ $D((H_1,\{u_1^1\}),(H'_1,\{v_1^1\}))<1$. That finishes the first step of the induction.

Next, set $v_1^2=v_1^1$ and $v_2^2=z_2$. Let $H'_2\leq G$ be an arbitrary finite group containing both $H'_1$ and $G_2$, e.g. the subgroup generated by these two groups. Again by Lemma \ref{rational_approx_lem} there exists $(H'_2,\{x_1,x_2\})\in\Age$ such that $D((H'_2,\{v_1^2,v_2^2\}),(H'_2,\{x_1,x_2\})<1/4$. By the $D$-property of $\alpha$, Fact \ref{glue_fact} and also Remark \ref{remark_converse} we can find $H_1\leq H_2\leq G$ isomorphic to $H'_2$ (via some $\phi_2$ extending $\phi_1$) and $u_1^2,u_2^2\in\Rat\Ur$ such that $d(u_1^1,u_1^2)<1/2$ and $D((H_2,\{u_1^2,u_2^2\}),(H'_2,\{v_1^2,v_2^2\}))<1/2$. This finishes the second step of the induction.\\

\noindent {\bf The general odd and even step of the induction.} The general steps are treated analogously as the second step of the induction. So we only briefly show the general odd $n$-th step of the induction, i.e. $n$ is now odd greater than $2$. For $i<n$ we set $u_i^n=u_i^{n-1}$ and we set $u_n^n=z_n$. Let $H_n\leq G$ be an arbitrary finite subgroup containing both $H_{n-1}$ and $G_n$. By Lemma \ref{rational_approx_lem} there exists $(H_n,\{x_1,\ldots,x_n\})\in\Age$ such that $D((H_n,\{u_1^n,\ldots,u_n^n\}),(H_n,\{x_1,\ldots,x_n\})<1/2^n$. By the $D$-property of $\beta$, Fact \ref{glue_fact} and also Remark \ref{remark_converse} we can find $H'_{n-1}\leq H'_n\leq G$ isomorphic to $H_n$ (via some $\phi_n$ extending $\phi_{n-1}$) and $v_1^n,\ldots,v_n^n\in\Rat\Ur$ such that $d(v_i^{n-1},v_i^n)<1/2^{n-1}$, for all $i<n$, and\\ $D((H_n,\{u_1^n,\ldots,u_n^n\}),(H'_n,\{v_1^n,\ldots,v_n^n\}))<1/2^{n-1}$. That finishes the inductive construction and the whole proof of Theorem \ref{thm_generic}.\\

In \cite{MeTs}, Melleray and Tsankov ask whether there is a Polish group $\mathbb{H}$ and a countable abelian group $G$ such that for a generic element $\alpha\in\Hom(G,\Iso(\Ur))$ the closure of $\alpha[G]$ in $\Iso(\Ur)$ is topologically isomorphic to $\mathbb{H}$. The most interesting case is when $G=\Int$, however the results from \cite{MeTs} suggest that the choice of $G$ is not important. That means, it seems likely that if the result holds for an unbounded countable abelian group $G$, then it holds for $\Int$ as well.

We show that the Melleray-Tsankov's problem has an affirmative answer in the non-abelian case when $G$ is the Hall's group.
\begin{prop}\label{prop_isom}
Let $H$ be a countable group and $X$ a Polish metric space. Suppose that $\beta,\gamma\in \Hom(H,\Iso(X))$ are weakly equivalent. Then the topological groups $\overline{\beta[H]}$ and $\overline{\gamma[H]}$ are topologically isomorphic.
\end{prop}
\begin{proof}
Let $(\phi,\psi)$ be a pair of a group automorphism and an autoisometry that witnesses that $\beta$ and $\gamma$ are weakly equivalent. Let $N_\beta$, resp. $N_\gamma$ be the kernels of $\beta$, resp. $\gamma$. It is easy to see that $\phi[N_\gamma]=N_\beta$, therefore $\phi$ induces an automorphism $\phi':H/N_\gamma\rightarrow H/N_\beta$. $H/N_\gamma$, resp. $H/N_\beta$ are countable dense subgroups of $\overline{\gamma[H]}$, resp. of $\overline{\beta[H]}$. We claim that $\phi'$ induces also a topological group isomorphism $\phi'':\overline{\gamma[H]}\rightarrow \overline{\beta[H]}$. It suffices to show that $\phi'$ is a homeomorphism between $H/N_\gamma$ and $H/N_\beta$ with the topologies inherited from $\Iso(X)$. We show that it is continuous. Showing that the inverse is continuous is analogous. Fix some $h\in H\setminus N_\gamma$ and some neighborhood $U$ of $[\phi(h)]_{N_\beta}$ which is given by some $x_1,\ldots,x_n\in X$ and $\varepsilon>0$ (viewing the equivalence class $[\phi(h)]_{N_\beta}$ as an isometry of $X$). However, then the neighborhood $V$ of $h$ (or $[h]_{N_\gamma}$) given by $\psi^{-1}(x_1),\ldots,\psi^{-1}(x_n)$ and $\varepsilon$ is such that $\phi'[V]=U$. That follows immediately from the definition of weak equivalence.
\end{proof}
The following is now an immediate consequence of Theorem \ref{thm_generic} and Proposition \ref{prop_isom}.
\begin{cor}
Let $G$ be the Hall's group. There exists a Polish group $\mathbb{H}$ such that for comeager many $\alpha\in\Hom(G,\Iso(\Ur))$ we have $\overline{\alpha[G]}\cong \mathbb{H}$.
\end{cor}
\begin{question}
What is $\mathbb{H}$ from the corollary?
\end{question}
Melleray and Tsankov pointed out to us that most likely $\mathbb{H}$ is going to be the full isometry group $\Iso(\Ur)$.

\section{Open problems about universal actions on Banach spaces}
In the last section we discuss universal actions by isometries on Banach spaces.

Recall that by the theorem of Mazur and Ulam (see e.g. Theorem 14.1 in \cite{BeLi}) every (onto) isometry on a Banach space is affine, that means it is a linear isometry plus translation. From that, one can derive that every action $\alpha: G\curvearrowright X$ of some group $G$ on a Banach space $X$ by isometries is determined by an action $\alpha_0:G\curvearrowright X$, which is by \emph{linear} isometries, and by a \emph{cocycle map} $b:G\rightarrow X$ which determines the corresponding translates. That is, for any $g\in G$ and $x\in X$ we have $$\alpha(g)x=\alpha_0(g)x+b(g).$$ Conversely, whenever we have an action $\alpha_0: G\curvearrowright X$ by linear isometries and a map $b:G\rightarrow X$ satisfying the so-called `cocycle condition', i.e. for every $g,h\in G$, we have $$b(gh)=\alpha_0(g)b(h)+b(g),$$ we can get an action of $G$ on $X$ by affine isometries. We refer the reader to Chapter 6 of \cite{NoYu} for more information.

This implies that the natural notion of embedding between two actions of groups on Banach spaces involves a group monomorphism and an equivariant affine isometric embedding. That motivates the following question.
\begin{question}
Does there exist an action of a countable locally finite group $G$ on a separable Banach space $X$ by affine isometries such that for every action of a locally finite countable group $H$ on a separable Banach space $Y$ by isometries there is a subgroup $H'\leq G$, isomorphic to $H$, and an affine isometric embedding $\phi:Y\hookrightarrow X$ that is, after identifying $H$ and $H'$, $H$-equivariant?
\end{question}

If one considers a less natural type of embeddings between actions of groups on Banach spaces by isometries, namely just equivariant isometric embeddings (thus treating Banach spaces just as metric spaces), the universality problem has a positive answer.

Let $F$ be the functor which sends a pointed metric space $(X,0)$ to its Lipschitz-free Banach space $F(X)$. By functoriality, any autoisometry of $X$ which preserves $0$ extends to a linear autoisometry of $F(X)$. However, even every autoisometry of $X$ extends to an affine autoisometry of $F(X)$. Indeed, let $\phi:X\rightarrow X$ be some autoisometry. In what follows, we view $X$ as a metric subspace of $F(X)$, i.e. we view every point $x\in X$ as a point in $F(X)$ also. Then the map $x\to \phi(x)-\phi(0)$ from $X$ to $F(X)$ is an isometric embedding of $X$ into $F(X)$ which preserves $0$, thus extends to a linear isometric embedding from $F(X)$ into $F(X)$. It is easy to check that it is actually onto. Composing it with the translation `$+\phi(0)$' gives the affine autoisometry that extends $\phi$.

Moreover, it is easy to check that by the same method every action of a group $G$ on a pointed metric space $(X,0)$ by isometries extends to the action of $G$ on $F(X)$ by affine isometries; i.e. the cocycle condition is satisfied. Thus from Theorem \ref{main_thm} we get the following corollary. We refer to Chapter 5 in \cite{Pe} for information about the Holmes space.
\begin{cor}
There exists an action of the Hall's group $G$ on the Holmes space $F(\Ur)$, the Lipschitz-free Banach space over the Urysohn space, such that for any action of a countable locally finite group $H$ on a separable Banach space $Y$ by isometries there exists $H'\leq G$ isomorphic to $H$ and an isometric embedding $Y$ into $F(\Ur)$ which is $H$-equivariant, after identifying $H$ and $H'$.
\end{cor}

Finally, one can consider actions of groups on Banach spaces by linear isometries; in other words, representations of groups in linear isometry groups of Banach spaces.

Below we propose a possible way how to prove an analogue of Theorem \ref{intro_mainthm} for such representations of locally finite groups by similar methods, again using Fra\" iss\' e theory. The proofs seem to be much more technical than in the case of plain metric spaces without algebraic structure. We suggest a class of actions of finite groups on finite-dimensional Banach spaces by isometries. We do not have a full proof of the amalgamation property for this class. However, provided it does exist it follows there is a universal representation of the Hall's group in the Gurarij space.\\

Let $F$ be a finite group and $I$ a non-empty finite set. By $F_I$ we denote the finite set $F\times I=\{x_{g,i}:g\in F,i\in I\}$. Instead of $x_{0,i}$, where $i\in I$ and $0\in F$ is the group zero, we may just write $x_i$. Consider now a finite-dimensional real vector space $E_{F,I}$ with $F_I$ as a basis. The canonical action of $F$ on $F_I$, where $g\cdot x_{f,i}=x_{gf,i}$ (or the permutation representation of $F$ on $F_I$), extends to a linear action of $F$ on $E_{F,I}$ (resp. the representation of $F$ in $\GL(E_{F,I})$).

Now let $W\subseteq E_{F,I}$ be any finite subset satisfying:
\begin{itemize}
\item $0\in W$; if $w\in W$, then $-w\in W$;
\item for every $i\neq j\in I$, $x_i-x_j\in W$;
\item for every $i\in I$, $g\in F$, $x_{g,i}\in W$;
\item for any $g\in F$ and $w\in W$, $g\cdot w\in W$.

\end{itemize}

A partial $F$-norm $\Norm'$ on $W$ is a partial norm on the finite set $W$ compatible with the action of $F$; that is, a function satisfying

\begin{itemize}
\item $\|w\|'=0$ iff $w=0$; (positivity);
\item $\|\alpha w\|'=|\alpha|\|w\|'$ provided that $w,\alpha w\in W$, for $\alpha\in \Rea$; (homogeneity)
\item $\|w\|'\leq \sum_{i=1}^n |\alpha_i|\|w_i\|'$, where $w=\sum_{i=1}^n \alpha_i w_i$ $w,(w_i)_{i=1}^n\subseteq W$, $(\alpha_i)_{i=1}^n\subseteq \Rea$; (triangle inequality)
\item $\|w\|'=\|g\cdot w\|'$, for $g\in F$, $w\in W$. (compatibility with the action)

\end{itemize}

Having $\Norm'$ we define a norm $\Norm$ on $E_{F,I}$ as the maximal extension of $\Norm'$ to the whole $E_{F,I}$. That is, for any $x\in E_{F,I}$ we set $$\|x\|=\min\{\sum_{j=1}^n |\beta_j|\|w_j\|':x=\sum_{j=1}^n \beta_j w_j,(w_j)_{i\leq n}\subseteq W\}.$$ It follows by compactness (from the finite-dimensionality) that the minimum is indeed attained.

Now it is straightforward to check that $\Norm$ is a norm that extends $\Norm'$ and that the action of $F$ on $E_{F,I}$ with $\Norm$ is by linear isometries. Notice that there are several other equivalent ways how to define $\Norm$ using $\Norm'$. For instance one can take the closed convex hull of the set $\{w/\|w\|':w\in W\}$ and then consider the Minkowski functional of such a set. The resulting norm will be $\Norm$. Another way is to consider the following set of functions $\mathcal{F}=\{f:F_I\cup\{0\}\rightarrow \Rea:f(0)=0,|\tilde{f}(w)|\leq \|w\|'\}$, where $\tilde{f}$ is the unique linear extension of $f$ on $E_{F,I}$. Then we have, for every $x\in E_{F,I}$, $\|x\|=\sup_{f\in\mathcal{F}} |\tilde{f}(x)|$.\\

We shall call such an action of $F$ on such a finite-dimensional space \emph{finitely presented}. If the partial norm $\Norm'$ is defined only on linear combinations of basis vectors with rational coefficients and it has a rational range, we shall call such a finitely presented action \emph{rational}.

Let us have finite-dimensional spaces $E_{F,I}$ and $E_{H,J}$, where $F,H$ are finite groups and $I,J$ finite sets. Suppose there are embeddings $\phi:F\hookrightarrow H$ and $\psi:I\hookrightarrow J$. Then they naturally induce a linear embedding of $E_{F,I}$ into $E_{H,I}$ which is also, after identifying $F$ and $\psi[F]\leq H$, $F$-equivariant. If $E_{F,I}$, resp. $E_{H,J}$ are equipped with the finitely presented norm, invariant by the action, and the linear embedding given by $\phi$ and $\psi$ is also isometric, we call such a pair $(\phi,\psi)$ an embedding between two finitely presented actions.

\begin{question}
Does the class of all finitely presented rational actions have the amalgamation property?
\end{question}

If the answer is affirmative, the Fra\" iss\' e limit would be an action of the Hall's group $G$ on a normed vector space $E_{G,J}$ with the corresponding extension property such that $\mathbb{G}$, the completion of $E_{G,J}$, is isometric to the Gurarij space. Similar arguments as in the proof of Theorem \ref{main_thm} show that it is a universal action within the class of actions of countable locally finite groups on separable Banach spaces by linear isometries.\\

 We conclude with few more open questions.
\begin{question}
Does there exist a universal unitary representation of a universal locally finite group?
\end{question}

Of special interest is whether one can amalgamate unitary representations.
\begin{question}
Let $G_1$ and $G_2$ be two groups with a common subgroup $G_0$. Let $\pi_1$ and $\pi_2$ be two unitary representations of $G_1$, resp. $G_2$, on Hilbert spaces $H_1$, resp. $H_2$ with a common subspace $H_0$, on which the restrictions $\pi_1\upharpoonright G_0$ and $\pi_2\upharpoonright G_0$ are equivalent. Does there exist an amalgam of these representations? Does this problem have a positive answer when the groups and spaces in question are finite, resp. finite-dimensional?
\end{question}


\end{document}